\documentclass[12pt]{amsart}

\usepackage[all]{xy}
\usepackage{enumerate}

\usepackage{amsmath, amssymb, amsthm}

\usepackage{graphicx}%
\usepackage{geometry}
\usepackage{mdframed}
\usepackage{multirow}%
\usepackage{amsmath,amssymb,amsfonts}%
\usepackage{amsthm}%
\usepackage{mathrsfs}%
\usepackage[title]{appendix}%
\usepackage{xcolor}%
\usepackage{textcomp}%
\usepackage{manyfoot}%
\usepackage{booktabs}%
\usepackage{algorithm}%
\usepackage{algorithmicx}%
\usepackage{algpseudocode}%
\usepackage{caption}
\usepackage{listings}%

\usepackage{enumitem}
\usepackage[normalem]{ulem} 

\usepackage[bookmarksnumbered=true]{hyperref} 
\usepackage[normalem]{ulem}
\hypersetup{
     colorlinks = true,
     linkcolor = blue,
     anchorcolor = blue,
     citecolor = blue,
     filecolor = blue,
     urlcolor = blue
}
\usepackage{url}
\usepackage{graphicx}%
\usepackage{multirow}%
\usepackage{amsmath,amssymb,amsfonts}%
\usepackage{amsthm}%
\usepackage{mathrsfs}%
\usepackage[title]{appendix}%
\usepackage{xcolor}%
\usepackage{textcomp}%
\usepackage{manyfoot}%
\usepackage{booktabs}%
\usepackage{algorithm}%
\usepackage{algorithmicx}%
\usepackage{algpseudocode}%
\usepackage{listings}%

\newtheorem{theorem}{Theorem}[section]
\newtheorem{corollary}[theorem]{Corollary}
\newtheorem{proposition}[theorem]{Proposition}
\newtheorem{lemma}[theorem]{Lemma}

\theoremstyle{definition}

\theoremstyle{remark}
\newtheorem{remark}[theorem]{Remark}
\numberwithin{equation}{section}

\title{Enumerating finite O-sequences: sub-Fibonacci behavior and growth estimates}

\author[F.~Cioffi]{Francesca Cioffi}
\address{Dip.~di Matematica e Appl. \\ Universit\`a degli Studi di Napoli Federico II\\ Via Cintia \\ 80126 Napoli \\ Italy.}
\email{\href{mailto:cioffifr@unina.it}{cioffifr@unina.it}}
\email{\href{mailto:maguida@unina.it}{maguida@unina.it}}

\author[M.~Guida]{Margherita Guida}

\author[E.~Pirozzi]{Enrica Pirozzi}
\address{Dipartimento di Matematica e Fisica\\ Università degli Studi della Campania Vanvitelli\\ viale Lincoln, 5 \\ 81100 Caserta\\Italy.}
\email{\href{mailto:enrica.pirozzi@unicampania.it}{enrica.pirozzi@unicampania.it}}


\begin{document}

\begin{abstract}Let $O_d$ denote the number of finite $O$-sequences of multiplicity $d$, namely the Hilbert functions of standard graded Artinian quotients of polynomial rings over a field. Starting from an iterative formula for computing $O_d$, we pursue two complementary directions. 
First, letting $A_d$ be the number of the finite $O$-sequences of multiplicity~$d$ whose last non-zero element is strictly larger than~$1$, we prove that the sequence $(A_{d+2})_{d\geq 1}$ is sub-Fibonacci. This result gives an enhancement of the sub-Fibonacci behavior of $(O_d)_{d\geq 1}$. Then, we provide a new algorithm for computing $O_d$, with more efficient performances than other available algorithms. We use the computed data and statistical methods to obtain an empirical calibration, in the interval $1\leq d \leq 1100$, of the Stanley-Zanello asymptotic upper bound for $\log(O_d)$ that better fits the observed values of $\log(O_d)$. An analogous study of the Stanley-Zanello asymptotic lower bound for $\log(O_d)$ is also carried out. The same method can be applied in every interval where the data are known. Some consequent prediction estimates are proposed. We also show that the sequence $(O_d/O_{d-1})_{d\geq 2}$ is strongly Cesàro convergent to~$1$. As a byproduct, we show that, if the sequence $(O_d/O_{d-1})_{d\ge 2}$ converges, then its limit must be equal to $1$, thereby giving a negative answer to a question posed by L. G. Roberts in 1992 under the assumption of convergence.
\end{abstract}



\maketitle

\section*{Introduction}

Finite O-sequences are the Hilbert functions of standard graded Artinian $K$-algebras over a field~$K$, and their combinatorics is governed by growth conditions (see \cite{Ma,H66}), with extremal behaviour realized by lex-segment ideals. Despite this structural description, quantitative problems such as enumerating finite O-sequences of a given multiplicity and analogous questions remain difficult \cite{MC, CG, CLM, ES, RR, SL, SZ}. 
For a positive integer $d$, let $O_d$ denote the number of finite O-sequences of multiplicity~$d$.  

Recently, in \cite{CG} it was shown that the sequence $(O_d)_d$ is sub-Fibonacci, according to an extension of the definition given by P.~C. Fishburn and F.~S. Roberts in \cite{FR}. Moreover it was observed that every term of the sequence $(O_d/O_{d-1})_d$ with $d\geq 6$ is bounded above by the golden ratio. 
This analysis and an explicit computation of the $A_d$ $O$-sequences of multiplicity $d$ with last non-zero value strictly larger than $1$ also produce an elementary method for computing~$O_d$, giving its values up to $d=60$. In addition, an iterative formula for $O_d$ is obtained by exploiting a decomposition of lex-segment ideals introduced by S.~Linusson in \cite{SL}. 
Starting from this iterative formula, in this paper we pursue two complementary directions.

First, we prove that the sequence $(A_{d+2})_{d\ge 1}$ is sub-Fibonacci, hence  giving an enhancement of the sub-Fibonacci behavior of $(O_d)$, since $O_d=O_{d-1}+A_d$, for every $d\geq 1$.

Second, we turn the iterative formula into an effective computation by encoding the required cardinalities via truncated generating functions and updating them layer-by-layer while storing only necessary intermediate data. This yields a memory-efficient algorithm that allows us to compute $O_d$ for all $1\leq d\leq 1100$. The resulting dataset enables an experimental study of the asymptotic upper and lower bounds of $\log(O_d)$ given in \cite{SZ} by R. Stanley and F. Zanello. We propose an empirical finite-range calibration of the upper bound that preserves the ${\sqrt{d}} \log(d)$ growth but substantially reduces the distance from the observed values of $\log(O_d)$ in the interval $1\leq d\leq 1100$. An analogous study of the Stanley-Zanello asymptotic lower bound for $\log(O_d)$ is also carried out. Using the Stanley-Zanello upper bound, we also show that the sequence $(O_d/O_{d-1})_{d\geq 2}$ is strongly Cesàro convergent to~$1$. As a consequence, we obtain that, if the sequence $(O_d/O_{d-1})_{d\geq 2}$ converges, then its limit is equal to $1$, giving a negative answer to the question posed by L. G. Roberts in \cite{RR}, which asked if the limit of $(O_d/O_{d-1})_{d\geq 2}$ is strictly larger than $1$ (see Section~\ref{sec:consequences}).

The paper is organized as follows. Section~\ref{sec:prelim} recalls background on O-sequences. Section~\ref{sec:iterative formula} develops the iterative formula and its algorithmic presentation, and establishes the sub-Fibonacci result for $(A_{d+2})_{d\ge 1}$. In Section \ref{sec:consequences}, as consequences of the upper bound proposed by Stanley and Zanello, we obtain the strongly Cesàro convergence of $(O_d/O_{d-1})_{d\geq 2}$ to $1$ and hence the negative answer to the question of Roberts. Section~\ref{sec:calibration} presents the calibration procedure, the numerical evidence up to $D=1100$ and some prediction estimates derived from this analysis.

\section{Preliminaries}
\label{sec:prelim}

Let $R:=K[x_1,\dots,x_p]$ be the polynomial ring over a field $K$ with the variables ordered as $x_1<\dots<x_p$. A term of $R$ is a power product $x^\alpha=x_1^{\alpha_1}\dots x_p^{\alpha_p}$, with $\alpha_i\in \mathbb Z_{\geq 0}$ for every $i\in\{1,\dots,p\}$. 

Given two positive integers $a$ and $t$, the {\em binomial expansion of $a$ in base $t$} is the unique writing
$$a  :=  \binom{k(t)}{t} + \binom{k(t-1)}{t-1} + \dots + \binom{k(j)}{j},
$$
where $k(t)> k(t-1)>\dots > k(j)\geq j\geq 1$, and with the convention that a binomial coefficient $\binom{n}{m}$ is null whenever either $n<m$ or $m<0$ 
and $\binom{n}{0}=1$, for all $n\geq 0$. Let 
$$
a^{\langle t\rangle}:=\binom{k(t)+1}{t+1} + \binom{k(t-1)+1}{t} + \dots + \binom{k(j)+1}{j+1}.
$$

A numerical function $h=(h_0,h_1,\dots,h_t,h_{t+1},\dots)$ is an {\em $O$-sequence} if it is the Hilbert function $H_{R/I}: t\in \mathbb Z_{\geq 0}\to \dim(R_t/I_t) \in \mathbb Z_{\geq 0}$ of a standard graded $K$-algebra $R/I$, for a homogeneous ideal $I$ (we also say that $H_{R/I}$ is the Hilbert function of $I$). Equivalently, $h=(h_0,h_1,\dots,h_t,h_{t+1},\dots)$ is an $O$-sequence if, and only if, $h_0=1$ and for every $t\geq 1$
\begin{equation}\label{eq:condition}
h_{t+1} \leq  h_t^{\langle t\rangle}.
\end{equation}
 
An {\em order ideal} $S$ is a set of terms closed under division.  The terms outside $S$ generate a monomial ideal $J$ for which $H_{R/J}(t)$ is the cardinality of $S_t$ because $S_t$ is a basis of the $K$-vector space $R_t/J_t$.

The ideal $J$ is a {\em lex-segment ideal} if, for every degree $t$ and for every term $\tau \in S_t$, all terms of degree $t$ lower than~$\tau$ with respect to the lexicographic term order belong to $S$. As independently studied in \cite{Ma} and \cite{H66}, the value \eqref{eq:condition} depends on the growth of the lex-segment ideals. In fact, for every $O$-sequence $h$, there is a unique lex-segment ideal $J$ with Hilbert function $h$ (the converse is straightforward).

In this paper we consider {\em finite} $O$-sequences $h=(h_0,h_1,\dots,h_s)$, where $s:=\max\{t : h_t\not=0\}$ is the {\em socle degree} of $h$ and $d:=\sum_{t=0}^s h_t$ is the {\em multiplicity or length} of $h$. Equivalently, we consider finite order ideals $M$ such that the corresponding monomial ideals $J$ are lex-segment ideals. 

From now on, for every positive integer d, we denote by $O_d$ the number of all finite $O$-sequences of multiplicity $d$.
Analogously, we denote by $A_d$ the number of all finite $O$-sequences of multiplicity~$d$ such that the last non-zero value is strictly larger than~$1$.

We recall that $O_d=O_{d-1}+A_d$, for every $d\geq 2$ (see \cite[Lemma 2.1]{CG}).

\section{Computation of $O_d$ by an iterative formula}
\label{sec:iterative formula}

In the paper \cite{CG} two different methods are proposed to compute $O_d$. The first is an elementary method based on an explicit computation of the $A_d$ $O$-sequences whose last non-zero value is strictly larger than~$1$ (see \cite[Section 2]{CG}).  The second method is based on an iterative formula (see \cite[Section 3]{CG}) on which here we focus. 


\subsection{Description of the iterative formula and sub-Fibonacci behavior of $A_d$}

In this section we refer to \cite{SL} and \cite{CG}. When necessary, we regard every finite $O$-sequence $h=(h_0,\dots,h_s)$ as extended by setting $h_i=0$, for all $i>s$.

For every integer $p>0$, $n\geq 0$, $k\geq 0$, $d>0$, $S(p,n,k,d)$ denotes the set of all order ideals of Artinian lex-segment ideals in at most $p$ variables, corresponding to finite $O$-sequences $h=(h_0,h_1,\dots,h_s)$ of multiplicity~$d$, satisfying the following conditions:

$\bullet$ the socle degree $s$ is at most $n$,

$\bullet$ $h_i=\binom{p-1+i}{i}$ for all $0\leq i\leq k$, and 

$\bullet$ $h_i<\binom{p-1+i}{i}$ for all $k<i\leq s$. 
\vskip 1mm
\noindent In other words, $S(p,n,k,d)$ consists of the order ideals whose Hilbert function coincides with that of the polynomial ring in $p$ variables up to degree $k$, and is strictly smaller afterwards. Let $O(p,n,k,d):=\vert S(p,n,k,d)\vert$ denote the cardinality of $S(p,n,k,d)$. 

For every $n\geq 0$, we have $O(0,n,0,1)=1$ and $O(0,n,k,d)=0$, for every $k>0$ or $d>1$. For $p<0$ or $n<0$ or $k<0$ or $d\leq 0$ we agree that $O(p,n,k,d)=0$.

\begin{lemma}\label{lemma:relazioni}
\
\begin{itemize}
\item[(i)] $O(p,n,k,d)\geq O(p,n,k,d-1)$, for every $p\geq 2$, $n\geq d-1$, $k\geq 1$, $d>2$.
\item[(ii)] $O(p,n,k,d) \geq O(p,n-1,k,d)$, for every $n$.
\item[(iii)] $O(p,n,d-1,d)=0$ for every $p>1$ and $d>1$. 
\item[(iv)] $O(p,n,k,d)=0$, for every $k>\min\{n,d-1\}$.
\item[(v)] For every $p$, $k$, $d$, if $n\geq d-1$, then $O(p,n,k,d)=O(p,d-1,k,d)$.
\end{itemize}
\end{lemma}

\begin{proof}
For item (i), if $h=(h_0,h_1,\dots,h_s)$ is the Hilbert function of an order ideal contained in $S(p,n,k,d-1)$, then $h'=(h_0,h_1,\dots,h_s,1)$ is contained in $S(p,n,k,d)$. Indeed, if $p\geq 2$, then the integer $k$ does not change in $h'$ and the bound $n\geq d-1$ is large enough for $h'$ too.

Items (ii)-(iv) follow straightforwardly by definition.
 
For item (v) it is enough to observe that every finite $O$-sequence of multiplicity $d$ has socle degree $s$ strictly lower than $d$. 
\end{proof}

\begin{lemma}\cite[Lemma 4.2]{CG}\label{lemma: base}
\
\begin{enumerate}
\item[(i)] $O_d=O(d,d-1,0,d)$.
\item[(ii)] $O(1,n,k,d)= \left\{\begin{array}{cl}1, &\text{ if } k=d-1 \text{ and } n\geq d-1\\ 0, &\text{ otherwise}\end{array}\right.$.
\item[(iii)] $O(p,n,0,d)=\sum_{k=0}^{d-1} O(p-1,n,k,d)$, for every $p\geq 2$.
\end{enumerate}
\end{lemma}

\begin{theorem}\cite[Theorem 4.3]{CG}\label{th: iterative formula}
For every integer $p>1$, $n\geq 0$, $k> 0$, $d>0$,
\begin{equation}\label{eq:formula k}
O(p,n,k,d)=\sum_{i=k}^n \sum_{j=1}^{d-1}  O(p-1,n,i,d-j)\cdot O(p,i-1,k-1,j).
\end{equation}
\end{theorem}

It is easy to note that $O_1=O_2=1$ because there are only the $O$-sequence $h=(1)$ of multiplicity~$1$ and the $O$-sequence $h'=(1,1)$ of multiplicity~$2$. The following result highlights some aspects of formula~\eqref{eq:formula k} applied to the case of $O_d$, for $d>2$. Recall that $A_d$ denotes the number of the $O$-sequences $(h_0,h_1,\dots,h_s)$ of multiplicity $d$ with $h_s\geq 2$, hence $A_1=A_2=0$.

Recall that a non-decreasing integer sequence $(x_k)_k$, for which $x_1=x_2=1$, is {\em sub-Fibonacci} if $x_k\leq x_{k-1}+x_{k-2}$ for all $k\geq 3$ (see \cite[page 262]{FR} in the case of finite integer sequences).
The sequence $(O_d)_d$ is {\em sub-Fibonacci} (see \cite[Proposition~2.3, Corollary~2.4]{CG}). We will show that $(A_{d+2})_{d\geq 1}$ is sub-Fibonacci too.

\begin{lemma}\label{lemma:Od}
With the above notation, for every $d>2$
\begin{equation}\label{eq:Od2} 
O_d=1 + \sum_{\ell=1}^{d-2}\sum_{k=1}^{d-2} O(d-\ell,d-1,k,d).
\end{equation}
\end{lemma}

\begin{proof}
By repeated use of Lemma \ref{lemma: base} we obtain
$$O_d=O(d,d-1,0,d)= \sum_{k=0}^{d-1} O(d-1,d-1,k,d)=O(d-1,d-1,0,d)+\sum_{k=1}^{d-1} O(d-1,d-1,k,d)$$
$$=O(d-2,d-1,0,d)+\sum_{k=1}^{d-1} O(d-2,d-1,k,d)+\sum_{k=1}^{d-1} O(d-1,d-1,k,d)=$$
$$=O(1,d-1,0,d)+\sum_{\ell=1}^{d-1}\sum_{k=1}^{d-1} O(d-\ell,d-1,k,d).$$
Then, observing that $O(1,d-1,0,d)=0$ and extracting the case $\ell=d-1$, we have
$$O_d=\sum_{\ell=1}^{d-2}\sum_{k=1}^{d-1} O(d-\ell,d-1,k,d)+ \sum_{k=1}^{d-1} O(1,d-1,k,d),$$
and observing that $\sum_{k=1}^{d-1} O(1,d-1,k,d)=1$ and extracting the case $k=d-1$ from the first summation, we obtain
$$O_d=1+ \sum_{\ell=1}^{d-2}\sum_{k=1}^{d-2} O(d-\ell,d-1,k,d)+\sum_{\ell=1}^{d-2} O(d-\ell,d-1,d-1,d).$$
To conclude the proof it is now sufficient to note that $\sum_{\ell=1}^{d-2} O(d-\ell,d-1,d-1 ,d)=0$ by Lemma \ref{lemma:relazioni}(iii), because $d-\ell\geq 2$.
\end{proof}

\begin{lemma}\label{lemma: Od bis}
With the above notation, for every $d>1$
\begin{equation}\label{eq:Od3}
O_d=\sum_{\ell=1}^{d-1}\sum_{k=1}^{d-1} \sum_{i=k}^{d-1}  O(d-\ell,d,i,d)\cdot O(d-\ell+1,i-1,k-1,1).
\end{equation}
\end{lemma}

\begin{proof}
We have
$$\sum_{\ell=1}^{d-1}\sum_{k=1}^{d-1} \sum_{i=k}^{d-1}  O(d-\ell,d,i,d)\cdot O(d-\ell+1,i-1,k-1,1)$$
$$=\sum_{\ell=1}^{d-1} \sum_{i=1}^{d-1}  O(d-\ell,d,i,d)=\sum_{\ell=1}^{d-1} \sum_{i=1}^{d-1}  O(d-\ell,d-1,i,d)$$
because $O(d-\ell+1,i-1,k-1,1)=1$ if, and only if, $k=1$ and is zero otherwise; moreover we observe that $O(d-\ell,d,i,d)=O(d-\ell,d-1,i,d)$ by Lemma~\ref{lemma:relazioni}(v). Then
$$=\sum_{\ell=1}^{d-2} \sum_{i=1}^{d-1}  O(d-\ell,d-1,i,d)+ \sum_{i=1}^{d-1}  O(1,d-1,i,d)$$
$$=\sum_{\ell=1}^{d-2} \sum_{i=1}^{d-2}  O(d-\ell,d-1,i,d)+1 + \sum_{\ell=1}^{d-2} O(d-\ell,d-1,d-1,d)$$
$$= \sum_{\ell=1}^{d-2} \sum_{i=1}^{d-2}  O(d-\ell,d-1,i,d)+1=O_{d}$$
by item (iii) of Lemma \ref{lemma:relazioni}, since $d>1$, and by \eqref{eq:Od2}, replacing $i$ with $k$.
\end{proof}

\begin{proposition}\label{prop:Od-Ad}
With the above notation, for every $d>2$
\begin{equation}\label{eq:Od4} 
O_d=
1+\sum_{\ell=1}^{d-2}\sum_{k=1}^{d-2} \sum_{i=k}^{d-2} \sum_{j=1}^{d-2} O(d-\ell-1,d-1,i,d-j)\cdot O(d-\ell,i-1,k-1,j),
\end{equation}
\begin{equation}\label{eq:Ad}
A_d= 1+\sum_{\ell=1}^{d-2}\sum_{k=1}^{d-2} \sum_{i=k}^{d-2} \sum_{j=2}^{d-2} O(d-\ell-1,d-1,i,d-j)\cdot O(d-\ell,i-1,k-1,j).
\end{equation}
\end{proposition}

\begin{proof}
For formula \eqref{eq:Od4}, it is enough to apply Theorem \ref{th: iterative formula} to every addend $O(d-\ell,d-1,k,d)$ in \eqref{eq:Od2}. Moreover, observe that the terms with $i=d-1$ or $j=d-1$ vanish by item (iv) of Lemma \ref{lemma:relazioni}, hence the upper limits in the sums may be replaced by $d-2$.

For formula \eqref{eq:Ad}, we start from \eqref{eq:Od4}. Recalling that $A_d=O_d-O_{d-1}$ by \cite[Lemma 2.1]{CG}, then we extract $O_{d-1}$ from \eqref{eq:Od4} observing that a factor in \eqref{eq:Od4} is a number of $O$-sequences of multiplicity $d-1$ if, and only if, $j=1$. Hence, we need to single out 
$$O(d-\ell-1,d-1,i,d-1).$$
So, we extract the case $j=1$ from the first summation of \eqref{eq:Od4} and keep it in $G(d)$,
$$O_d=1+\sum_{\ell=1}^{d-2}\sum_{k=1}^{d-2} \sum_{i=k}^{d-2} \sum_{j=2}^{d-2} O(d-\ell-1,d-1,i,d-j)\cdot O(d-\ell,i-1,k-1,j)+G(d)$$
where, by Lemma \ref{lemma: Od bis},
$$G(d)=\sum_{\ell=1}^{d-2}\sum_{k=1}^{d-2} \sum_{i=k}^{d-2}  O(d-\ell-1,d-1,i,d-1)\cdot O(d-\ell,i-1,k-1,1)=O_{d-1}.$$
Finally we obtain
$$O_d=1+\sum_{\ell=1}^{d-2}\sum_{k=1}^{d-2} \sum_{i=k}^{d-2} \sum_{j=2}^{d-2} O(d-\ell-1,d-1,i,d-j)\cdot O(d-\ell,i-1,k-1,j)+O_{d-1}.$$
Hence
$$A_d=1+\sum_{\ell=1}^{d-2}\sum_{k=1}^{d-2} \sum_{i=k}^{d-2} \sum_{j=2}^{d-2} O(d-\ell-1,d-1,i,d-j)\cdot O(d-\ell,i-1,k-1,j)$$
recalling that $O_d=O_{d-1}+A_d$ (see \cite[Lemma~2.1]{CG}).
\end{proof}

\begin{proposition}
$A_{d}\leq A_{d+1}$, for every $d>2$. In particular, the sequence $(A_{d+2})_{d\geq 1}$ is sub-Fibonacci.
\end{proposition}

\begin{proof}
We prove the inequality $A_{d+1}\geq A_d$ by isolating non-negative remainder terms in formula \eqref{eq:Ad}.
Extracting the case $\ell=1$ and putting it in~$R_1$ we obtain
$$A_{d+1}=1+\sum_{\ell=2}^{d-1}\sum_{k=1}^{d-1} \sum_{i=k}^{d-1} \sum_{j=2}^{d-1} O(d-\ell,d,i,d+1-j)\cdot O(d+1-\ell,i-1,k-1,j)+R_1.$$
Resetting $\ell=\ell-1$ and observing that $O(d-\ell,d,i,d+1-j)=O(d-\ell,d-1,i,d+1-j)$ by Lemma \ref{lemma:relazioni}(v), since $d-1\geq (d+1-j)-1=d-j$,
$$A_{d+1}=1+\sum_{\ell=1}^{d-2}\sum_{k=1}^{d-1} \sum_{i=k}^{d-1} \sum_{j=2}^{d-1} O(d-\ell-1,d-1,i,d+1-j)\cdot O(d-\ell,i-1,k-1,j)+R_1.$$
Observe that for $i=d-1$ the factors $O(d-\ell-1,d-1,i,d+1-j)$ are null and that for $k=d-1$ the factors $O(d-\ell,i-1,k-1,j)$ are null by item (iv) of Lemma \ref{lemma:relazioni} if $j<d-1$ and by item (iii) of Lemma \ref{lemma:relazioni} if $j=d-1$. Then, we extract the case $j=d-1$ and put it in $R_2$
$$A_{d+1}=1+\sum_{\ell=1}^{d-2}\sum_{k=1}^{d-2} \sum_{i=k}^{d-2} \sum_{j=2}^{d-2} O(d-\ell-1,d-1,i,d+1-j)\cdot O(d-\ell,i-1,k-1,j)+R_1+R_2.$$
We now observe that, for every $1\leq \ell<d-2$, $O(d-\ell-1,d-1,i,d+1-j)\geq O(d-\ell-1,d-1,i,d-j)$ by Lemma \ref{lemma:relazioni}(i) because $i>0$ and $d-\ell-1>1$. It remains to consider the case $\ell=d-2$.

For $\ell=d-2$, $O(d-\ell-1,d-1,i,d+1-j)=O(1,d-1,i,d+1-j)=1$ if $i=d-j$, and it is null otherwise. Analogously, $O(1,d-1,i,d-j)=1$ if $i=d-j-1$, and it is null otherwise. 

Nevertheless, for $\ell=d-2$, we have $O(2,d-j-1,k-1,j)\geq O(2,d-j-2,k-1,j)$, by Lemma \ref{lemma:relazioni}(ii). Thus, from the summation we extract, for $\ell=d-2$ (for which only $i=d-j$ gives a contribution),
$$\sum_{k=1}^{d-2} \sum_{i=k}^{d-2} \sum_{j=2}^{d-2} O(1,d-1,i,d+1-j)\cdot O(2,i-1,k-1,j)=\sum_{k=1}^{d-2} \sum_{j=2}^{d-2} O(2,d-j-1,k-1,j)$$
$$\geq \sum_{k=1}^{d-2} \sum_{j=2}^{d-2} O(2,d-j-2,k-1,j)$$
where the inequality follows from the argument above and the last summation appears in the writing \eqref{eq:Ad} of $A_d$ for $\ell =d-2$ (for which only $i=d-j-1$ gives a contribution). Finally
$$A_{d+1}\geq 1+\sum_{\ell=1}^{d-2}\sum_{k=1}^{d-2} \sum_{i=k}^{d-2} \sum_{j=2}^{d-2} O(d-\ell-1,d-1,i,d-j)\cdot O(d-\ell,i-1,k-1,j)= A_d.$$
To prove the last assertion, set $x_n:=A_{n+2}$ for every $n\geq 1$. Then $x_1=A_3=1$ and $x_2=A_4=1$. Moreover, for the first part of the statement, $(x_n)_n$ is non-decreasing and $x_n=A_{n+2} \leq A_{n+1}+A_n=x_{n-1}+x_{n-2}$, for every $n>2$ (see \cite[Lemma~2.2]{CG}). Hence, $(A_{d+2})_{d\geq 1}$ is a sub-Fibonacci sequence.
\end{proof}

\subsection{Algorithm and computational cost}

In this section we describe an algorithm (see Algorithm \ref{algorithm}) that implements the iterative formula~\eqref{eq:formula k}, combined with Lemma~\ref{lemma:relazioni} and Lemma~\ref{lemma: base}, in the particular case the formula gives the value of $O_d$ (see \cite[Algorithm 2]{CG} for an algorithm that describes a computation of formula \eqref{eq:formula k} in the general case). Hence, it computes the values of $O_d$ for all $1 \le d \le D$, for a prescribed bound $D$, and allowed us to compute $O_d$, for every $1\leq d\leq D=1100$ (a list of these values computed in Python and the implementation of Algorithm \ref{algorithm} that we used to make this computation are available at \url{https://github.com/FrancescaCioffi/OSequences1-1100}).

\begin{remark}
The availability of the values of $O_d$ up to $D=1100$ allows us to verify computationally that the sequence $(O_d/O_{d-1})_d$ is decreasing for every $12\leq d\leq 1100$. If $(O_d/O_{d-1})_d$ were decreasing for every $d\geq 12$, then it would be convergent and its limit would be $1$ as will be shown in Section~\ref{sec:consequences}.
\end{remark}

The quantities $O(p,n,k,d)$ are encoded by truncated generating functions and computed by increasing the parameter $p$ in the following way.
For $p \geq 1$ and $0 \leq k \leq n \leq D-1$, we define the truncated generating functions
\begin{equation}\label{eq:Fpnk}
  F_{p,n,k}(t)\ :=\ \sum_{d=1}^{D} O(p,n,k,d)\,t^d .
\end{equation}
By Lemma~\ref{lemma: base}(i), the value $O_d$ coincides with the coefficient of $t^d$ in $F_{d,d-1,0}(t)$.

The computation proceeds inductively for $p=1,\dots,D$, storing only the two consecutive layers corresponding to $p-1$ and $p$, which guarantees memory efficiency. After every step at $p=d$, the values $O_d$ are obtained from $F_{d,d-1,0}(t)$.

\medskip
\noindent\textbf{Initialization ($p=1$).}
By Lemma~\ref{lemma: base}(ii),
\[
O(1,n,k,d)=
\begin{cases}
1 &\text{if } k=d-1 \text{ and } n\ge d-1,\\
0 &\text{otherwise.}
\end{cases}
\]
Hence $F_{1,n,k}(t)=t^{k+1}$ whenever $k+1\le D$ and $n\ge k$, and $F_{1,n,k}(t)=0$ otherwise.

\medskip
\noindent\textbf{Inductive step ($p\ge 2$).}
Assume that all $F_{p-1,n,k}(t)$ have been computed. 

For $k=0$, Lemma~\ref{lemma: base}(iii) gives
\begin{equation}\label{eq: F0}
F_{p,n,0}(t)
=
\sum_{k=0}^{D-1} F_{p-1,n,k}(t),
\end{equation}
where $ F_{p-1,n,k}(t)=0$ if $k>n$.

For $k>0$, Theorem~\ref{th: iterative formula} yields
\[
F_{p,n,k}(t)
=
\sum_{i=k}^{n}
\mathrm{PolyMulTrunc}
\bigl(
F_{p-1,n,i}(t),
F_{p,i-1,k-1}(t),
D
\bigr),
\]
where $\text{PolyMulTrunc}(a, b,D)$ is a procedure that computes $F_{p,n,k}(t)$ by a convolution algorithm. More precisely, given two polynomials $a$ and $b$, for each coefficient $a_i$ of the first polynomial, the algorithm checks whether $a_i\neq 0$. If so, it multiplies $a_i$ by all coefficients $b_j$ of the second polynomial such that $i + j \leq D$, again skipping the terms where $b_j = 0$. Each product $a_i\cdot b_j$ is then added to the coefficient of degree $i + j$ in the result. Zero-coefficients are skipped to speed up. The correctness of the algorithm follows directly from Lemma~\ref{lemma:relazioni}, Lemma~\ref{lemma: base}, and Theorem~\ref{th: iterative formula}. 

Summarizing, for each $p=1,\dots,D$ and $n=0,\dots,D-1$, the algorithm computes first $F_{p,n,0}(t)$, then $F_{p,n,k}(t)$ for $k=1,\dots,n$. 
In the implementation only the truncated generating functions corresponding to $p-1$ and $p$ are stored in two arrays ensuring memory efficiency
$$\mathtt{prev}[n][k]:=F_{p-1,n,k}(t),  \quad \mathtt{curr}[n][k]:=F_{p,n,k}(t),$$
which are updated at the end of each iteration on $p$. In the pseudo-code that we exhibit we omit the use of $\mathtt{curr}$ and $\mathtt{prev}$ for simplicity.

\begin{algorithm}[!ht]
	\caption{\label{algorithm} Iterative computation of $O_d$ via truncated generating functions, for $1\leq d\leq D$}
	\begin{algorithmic}[1]
		\State IterativeFormula$\left(D\right)$
		\Require a positive integer $D$
		\Ensure the values $O_d$, for every $1\leq d\leq D$ 
		\For{$0\leq k\leq n\leq D-1$}
		{
		\If{$k+1\leq D$ and $n\geq k$}
		\State $F_{1,n,k}(t)\gets t^{k+1}$
        \Else
		\State $F_{1,n,k}(t)\gets 0$
        \EndIf}
        \EndFor
\State $O_1:=1$
\For{$p = 2$ \textbf{to} $D$}
\For{$n = 0$ \textbf{to} $D-1$}
\State       {\textbf{Case $k=0$:} 
\State           { {$F_{p,n,0}(t) \gets \sum_{h = 0}^{D-1} F_{p-1,n,h}(t)$}} }
\State      {\textbf{Case $k>0$:} 
       \For{$k = 1$ \textbf{to} $n$} 
\State            $F_{p,n,k}(t) \gets \sum_{i=k}^{n} \text{PolyMulTrunc}(F_{p-1,n,i}(t), F_{p,i-1,k-1}(t), D)$
        \EndFor
        \EndFor
    \State $O_p \gets$ coefficient of $t^p$ in $F_{p,p-1,0}(t)$
\EndFor\\
\Return{$(O_1,\dots,O_D)$}
}	
\end{algorithmic}  
\end{algorithm}

\medskip
\noindent\textbf{Complexity.}
Let $M(D)$ denote the cost of multiplying two polynomials truncated at degree $D$. For each fixed $p$, there are $O(D^2)$ pairs $(n,k)$ with $0\le k\le n\le D-1$. For $k>0$, each update of $F_{p,n,k}(t)$ involves $O(D)$ truncated products, so one performs $O(D^3)$ calls to $\text{PolyMulTrunc}(a,b,D)$ per layer.
Thus the overall running time is $O\!\big(D^4\,M(D)\big)$; with classical convolution we have $M(D)=O(D^2)$, which gives the worst-case bound $O(D^6)$, as number of arithmetic operations on integer coefficients.
Memory usage consists of storing two layers, i.e.\ $O(D^2)$ truncated polynomials of length $D+1$, for a total of $O(D^3)$ integer coefficients. 

\medskip
\noindent\textbf{Comparison with other algorithms.} Different computations of the integers $O_d$ are proposed in the papers \cite{ES} and \cite{CG}. In \cite{ES}, the first $20$ values of $O_d$ are computed based only on the Macaulay condition~\eqref{eq:condition} (also see \url{https://oeis.org/A232476}). In \cite{CG}, two algorithms are proposed. The first algorithm (see \cite[Algorithm 1]{CG}) is based on the Macaulay condition \eqref{eq:condition}, but also includes several actions aimed at reducing the required memory and time (see \cite[Remark 3.2]{CG}), obtaining the first~$60$ values of~$O_d$ (again, see \url{https://oeis.org/A232476}). Nevertheless, the computational cost of an algorithm based on the Macaulay condition~\eqref{eq:condition} is governed by the rapid growth of the number of finite $O$-sequences of a given multiplicity. 
The second algorithm (see \cite[Algorithm 2]{CG}) consists of a quite straightforward application of the iterative formula \eqref{eq:formula k}, in its generality. 

Here, Algorithm \ref{algorithm} is based on the iterative formula \eqref{eq:formula k} too, but significantly improves the performance and focuses on the computation of $O_d$. The enhancements that we have described above in detail essentially consist of the introduction of truncated generating functions and of a convolution multiplication between vectors, with checks for zero factors. Although in the worst-case the computational cost of Algorithm~\ref{algorithm} is the same as \cite[Algorithm 2]{CG}, in practice the effective cost is smaller because the procedure systematically skips zero coefficients, thereby exploiting the sparsity that frequently occurs in the intermediate generating functions, so allowing the computation of $O_d$, for every $d\leq D=1100$.

\begin{remark}
Alternatively to the routine $\mathtt{PolyMulTrunc}(a,b,D)$  which performs the classical quadratic convolution algorithm, truncated at degree~$D$, a Fast Fourier Transform (FFT) technique could be employed. However, the choice to use the classical quadratic convolution algorithm guarantees exact results.
\end{remark}

\section{The Stanley-Zanello upper and lower bounds and consequences}
\label{sec:consequences}

In \cite{SZ} R. Stanley and F. Zanello give upper and lower bounds for $O_d$ of the form $C_1\sqrt{d} \leq \log(O_d)\leq C_2{\sqrt{d}} \log(d)$, where $\log$ denotes the natural logarithm and $C_1$, $C_2$ are constants.  
 
We explicitly write an expanded version of the bounds proposed in \cite[Theorem~1]{SZ} for~$\log(O_d)$ using the notation of~\cite{CG}. 
By the proof of \cite[Theorem 1]{SZ}, for every integer $d>2$ one has
\begin{equation}\label{eq:SZ_start}
p(d-1)\le O_d \le \sqrt{2d} \, p(d)\,e^{\sqrt{2d}\log (d)},
\end{equation}
 where $p(d)$ denotes the integer partition function. By the classical Hardy--Ramanujan formula for the partition function
(see, e.g., \cite[Formula (1.41)]{HR}),
\begin{equation}\label{eq:HR}
p(m)\sim \frac{1}{4m\sqrt3}\, \exp \left(\pi\sqrt{\frac{2m}{3}}\right)
\qquad (m\to +\infty),
\end{equation}
where the symbol $\sim$ means that $p(m)=[\frac{1}{4m\sqrt3}\, \exp \left(\pi\sqrt{\frac{2m}{3}}\right)](1+o(1))$\footnote{For a function $f(x)$, here $f(x)=o(1)$ means that $\lim_{x\to +\infty}f(x)=0$.}. Therefore,
\begin{equation}\label{eq:logHR}
\log (p(m))=\pi\sqrt{\frac{2m}{3}}-\log(4m\sqrt3)+o(1)
\qquad (m\to +\infty).
\end{equation}
Combining \eqref{eq:SZ_start} and \eqref{eq:logHR}, as $d\to +\infty$, we obtain
\[
\log(O_d)\ge \log p(d-1)
= \pi\sqrt{\frac{2(d-1)}{3}}
 -\log\!\bigl(4(d-1)\sqrt3\bigr)+o(1), \quad \text{and}
\]
\[
\log(O_d)\le \frac{1}{2}\log(2d)+\log p(d)+\sqrt{2d}\log (d) \ 
= \ \sqrt{2d}\log (d)+\pi\sqrt{\frac{2d}{3}}
-\frac12\log(24d)+o(1).
\]
Hence, as $d\to +\infty$,
\begin{equation}\label{eq:bound completo}
\pi\sqrt{\frac{2(d-1)}{3}}
-\log\!\bigl(4(d-1)\sqrt3\bigr)+o(1)
\le \log(O_d)\le
\sqrt{2d}\log (d)+\pi\sqrt{\frac{2d}{3}}
-\frac12\log(24d)+o(1).
\end{equation}

\begin{remark}\label{rem:storico}
The study of the partition function has a long history and raised the attention of many important mathematicians. Here we would only like to recall that, after the generating function of $p(m)$ due to Euler, the study of partitions gradually moved from exact and recursive descriptions to asymptotic methods. A decisive step in this direction was the Hardy-Ramanujan formula, which revealed the leading exponential growth of $p(m)$, and was later complemented by Rademacher’s exact convergent series. 
\end{remark}

\begin{remark}\label{rem:tecnico}
The key point of the upper bound of Stanley and Zanello is that the behaviour of the $O$-sequence $h=(h_0,\dots,h_s)$ is truly \emph{critical} only in a certain range of degrees. More precisely, if $d$ is the multiplicity of $h$ and $j$ denotes the minimal index such that $h_j\leq j$, then the current proof of the upper bound in \cite{SZ} shows that $j\leq \sqrt{2d}$, allowing $d$ independent choices only in each degree up to $j$. In \cite[Remark~2]{SZ} some finer considerations are presented about the critical range of degrees to be taken.
\end{remark}

The asymptotic upper bound given in \eqref{eq:bound completo} has the following direct elementary consequences on the behaviour of the rate sequence $(O_d/O_{d-1})_{d\geq 2}$, even in terms of statistical convergence and strong Cesàro convergence. 

\begin{lemma}\label{lemma: conseguenza 1}
With the above notation, $\displaystyle{\lim_{d\to +\infty}}\frac{\log(O_d)}{d}=0$. In particular, $\displaystyle{\lim_{d\to +\infty}}({O_d}^{1/d})=~1$.
\end{lemma}

\begin{proof}
By \eqref{eq:bound completo}, there exist a constant $C>0$ and an integer $d_1$ such that $\log(O_d)\le C\sqrt d\,\log d$, for every $d\ge d_1$.
Moreover, since $O_d\ge 1$ for every $d\ge 1$, we have $\log(O_d)\ge 0$. Hence, for every $d\ge d_1$,
\[
0\le \frac{\log(O_d)}{d}\le C\,\frac{\log d}{\sqrt d}.
\]
Since
\[
\frac{\log d}{\sqrt d}\longrightarrow 0
\qquad\text{as } d\to +\infty,
\]
the squeeze theorem yields the first assertion.
The second assertion follows straightforwardly from the first assertion, because $\log({O_d}^{\frac{1}{d}})=\frac{1}{d}\log(O_d)$.
\end{proof}



The notion of statistical convergence was introduced by H. Fast in \cite{Fast}; a later treatment is due to J. A. Fridy \cite{Fridy}, according to whom a sequence $x=(x_k)$ is {\em statistically convergent} to $L$ if, for every $\varepsilon>0$, the set $\{k\in \mathbb{N}: |x_k-L|>\varepsilon\}$ 
has natural density zero\footnote{Let $A(x)$ be the number of positive integers in a set $\mathcal A$ that are less than or equal to $x$. In case the sequence $A(n)/n$ has limit, we say that $\mathcal A$ has {\em natural density} $\delta(\mathcal A):=\displaystyle{lim_{n\to +\infty}}\frac{A(n)}{n}$.}. Following Connor \cite{Connor}, a sequence $x=(x_k)$ is said to be {\em strongly $p$-Cesàro convergent} to $L$, for $p>0$, if
$\displaystyle{\frac{1}{N}\sum_{k=1}^{N}|x_k-L|^p}$ converges to~$0$, for $N\to +\infty$.
In the case $p=1$, this is simply called {\em strong Cesàro convergence}.

\begin{proposition}\label{prop: Cesaro}
The sequence $(\frac{O_d}{O_{d-1}})_d$ is strongly Cesàro convergent to $1$; namely, for $N\to +\infty$,
$$
\frac{1}{N-1}\sum_{d=2}^N \left|\frac{O_d}{O_{d-1}}-1\right| \longrightarrow 0.$$ 
Equivalently, it is statistically convergent to $1$; namely, for $N\to +\infty$, for every $\varepsilon>0$,
$$\frac{1}{N-1} \#\left\{2\leq d\leq N: \left|\frac{O_d}{O_{d-1}}-1\right|\geq \varepsilon \right\} \longrightarrow 0.$$
\end{proposition}

\begin{proof}
First, observe that for every $1\leq x\leq 2$, we have 
$$x-1 \leq x\log(x) \leq  2 \log(x).$$ 
Thanks to \cite[Proposition 2.3]{CG}, for every $d\geq 2$, 
$$1\leq \frac{O_d}{O_{d-1}}\leq 2.$$ 
Hence, we can replace $x$ by $\frac{O_d}{O_{d-1}}$ obtaining 
$$0 \leq \frac{O_d}{O_{d-1}}-1 \leq 2 \ \log\Bigl(\frac{O_d}{O_{d-1}}\Bigr).$$
We also have $O_N=O_1\prod_{d=2}^N \frac{O_d}{O_{d-1}}= \prod_{d=2}^N \frac{O_d}{O_{d-1}}$ because $O_1=1$. Then
$$0 \leq \frac{1}{N-1} \sum_{d=2}^N \Bigl(\frac{O_d}{O_{d-1}}-1\Bigr)  \leq 2 \ \frac{1}{N-1} \sum_{d=2}^N \log\Bigl(\frac{O_d}{O_{d-1}} \Bigr) = 2 \ \frac{\log(O_N)}{N-1}$$
and so the first assertion holds thanks to Lemma \ref{lemma: conseguenza 1}.

The equivalence with the statistical convergence follows from \cite[Theorem 2.1]{Connor} since the sequence $(O_d/O_{d-1})_{d\geq 2}$ is bounded, thanks to \cite[Proposition 3.6(i)]{CG}, as already highlighted. 
\end{proof}

Actually, we do not know if the sequence $(\frac{O_d}{O_{d-1}})_d$ converges in the usual sense. Nevertheless, the following result shows that, if the sequence $(\frac{O_d}{O_{d-1}})_{d\geq 2}$ converges, then $\lim_{d\to\infty} \frac{O_d}{O_{d-1}}=~1$. This gives a negative answer to a question posed by L.~G.~Roberts in~\cite{RR}, under the assumption of convergence. 

\begin{corollary}\label{cor:answer} {\rm[Answer to L. G. Roberts' question]} If the sequence $(\frac{O_d}{O_{d-1}})_{d\ge 2}$ converges, then its limit is equal to $1$.
\end{corollary}

\begin{proof}
Under the additional assumption that the sequence $(\frac{O_d}{O_{d-1}})_{d\ge 2}$ converges ordinarily to a limit $L$, then it is also strongly Ces\`{a}ro convergent to $L$. On the other hand, Proposition~\ref{prop: Cesaro} guarantees  that the sequence is strongly Ces\`{a}ro convergent to $1$. Then, the statement follows by uniqueness of the limit. 
\end{proof}

\begin{remark}
Since $\displaystyle \frac{1}{N-1}\sum_{d=2}^{N}\left(\frac{O_d}{O_{d-1}}-1\right) = \frac{1}{N-1}\sum_{d=2}^{N}\frac{O_d}{O_{d-1}}-1,$ 
Proposition \ref{prop: Cesaro} can be restated as 
\hspace{.2cm} $\displaystyle \frac{1}{N-1}\sum_{d=2}^{N}\frac{O_d}{O_{d-1}}\longrightarrow 1.$
Furthermore, since $\displaystyle \frac{O_d}{O_{d-1}}-1=\frac{A_d}{O_{d-1}}$ by \cite[Lemma~2.1]{CG}, the sequence $\displaystyle \left(\frac{O_d}{O_{d-1}}\right)_{d\ge2}$ converges ordinarily (respectively, is strongly Ces\`{a}ro convergent, and hence statistically convergent) to $1$, if and only if, the sequence $\displaystyle \left(\frac{A_d}{O_{d-1}}\right)_{d\ge2}$ converges ordinarily (respectively, is strongly Ces\`{a}ro convergent, and hence statistically convergent) to~$0$.    
\end{remark}


\section{Experimental estimates for the Stanley-Zanello upper and lower bounds in the interval $1\leq d\leq 1100$}
\label{sec:calibration}

The goal of this section is to produce empirical calibrations of the upper and lower bounds of Stanley and Zanello in the interval $1\leq d\leq 1100$.

First, we focus on an empirical calibration of the upper bound that preserves ${\sqrt{d}}\log(d)$ growth while staying closer to the observed values of $\log(O_d)$, in the finite range $1\leq d\leq 1100$. Then, we also observe that in the same interval the Stanley-Zanello lower bound for $\log(O_d)$ better fits the observed values of $\log(O_d)$ than the upper bound, according to the theoretical considerations presented in \cite[Remark 2]{SZ}.

All the computations reported in the following part of the manuscript were performed in~R. 

\medskip
\noindent{\bf Case D=1100: upper bound.}
For a positive integer $D$ and a discrete function $f(d)$, let $L_{f}(d) = \beta_{0,f} + \beta_{1,f} d$ be the linear function computed by the least-squares method
\footnote{We recall the formulas:  
	$\beta_{1,f}
		= \frac{\sum_{d=1}^D (d - \bar{d})(f(d) - \overline{f})}
		{\sum_{d=1}^D (d - \bar{d})^2}$, 
$\beta_{0,f} = \overline{f} - \beta_{1,f} \bar{d}$,
\ where $\bar{d} = \frac{(D+1)}{2}$, and $\overline{f} = \frac{1}{D}\sum_{d=1}^D f(d)$.} 
    that best fits the data points $(d,f(d))$ in the interval $1\leq d\leq D$. 
Then we consider the explicit part of the asymptotic upper bound in \eqref{eq:bound completo}, namely, 
	\begin{equation} \label{UPB}
	up(d):={\sqrt{2d}}\,\Bigl(\log(d)
	\;+\;
	\frac{\pi}{\sqrt{3}}\Bigr)
	\;-\;
	\frac{\log(24\,d)}{2}.
	\end{equation} 
From the relation
$$L_{\log(O_d)}=(L_{up(d)}-\beta_{0,up})\frac{\beta_{1,\log(O_d)}}{\beta_{1,up}}+\beta_{0,\log(O_d)}$$ 
we obtain the following discrete function
\begin{equation}\label{trasf}
		\widehat{up}(d):=(up(d)-\beta_{0,up})\frac{\beta_{1,\log(O_d)}}{\beta_{1,up}}+\beta_{0,\log(O_d)} \quad \text{ (affine rescaling)},
	\end{equation}
where the symbols $\beta_{0,up}$, $\beta_{1,up}$, $\beta_{0,\log(O_d)}$, $\beta_{1,\log(O_d)}$ are $\beta_{0,f}$ and $\beta_{1,f}$, with $f(d)=up(d)$ and $f(d)=\log(O_d)$, respectively.

The construction \eqref{trasf} of $\widehat{up}(d)$ does not exclude that there exists a positive integer~$d$ such that $\widehat{up}(d)-\log(O_d)<0$, namely $\widehat{up}(d)<\log(O_d)$ (see the second line of Table~\ref{tab:table2}). Hence, we also consider the function
		\begin{equation}\label{trasfbest}
		\widehat{\widehat{up}}(d):= \widehat{up}(d)+ \max_{d\leq D}\left(\log(O_d)-\widehat{up}(d)\right),
	\end{equation}
because by definition it satisfies $\log(O_d)\leq \widehat{\widehat{up}}(d)$, for every positive integer $d\leq D$. The function~\eqref{trasfbest} turns out to be an empirical calibration of the upper bound for $\log(O_d)$ that behaves better than the function $up(d)$ in the interval $1\leq d\leq D=1100$, as we show in the following.

{\em We highlight that both the functions $\widehat{up}(d)$ and $\widehat{\widehat{up}}(d)$ depend on the value of $D$, although we are choosing a notation that does not evidences it. When it will be necessary we make this dependence explicit even in the notation.}

As already observed, in the interval $1\leq d\leq D=1100$, there are some values $d$ such that $\widehat{up}(d)-\log(O_d)<0$ and so the function $\widehat{up}(d)$ is not an upper bound for $\log(O_d)$ even in the given interval. Nevertheless, it turns out to be a much better empirical approximation for the observed values of $\log(O_d)$ than $up(d)$ in the given interval, in the sense that 
	\begin{equation}\label{goodness}
		\max_{1\leq d\leq 1100}\vert\widehat{up}(d)-\log(O_d)\vert
		< \max_{1\leq d\leq 1100}\vert up(d)-\log(O_d)\vert,
	\end{equation}
as the data in Table~\ref{tab:table2} show.
Moreover, in the finite range $1\leq d\leq 1100$ one has
\[
\log(O_d)\le \widehat{\widehat{up}}(d) < up(d),
\]
where the first inequality follows from the definition of $\widehat{\widehat{up}}(d)$, while the second follows by a direct numerical check. Thus, in the same finite range $1\leq d\leq 1100$ 
	\begin{equation}\label{good}
		\max_{d\leq D} \left(\widehat{\widehat{up}}(d)-\log(O_d)\right)< \max_{d\leq D}({up(d)}-\log(O_d)).
	\end{equation}
Using \eqref{trasf} and \eqref{trasfbest}, we can write
\begin{equation}\label{linear_transUP}
\widehat{\widehat{up}}(d)= \left(\frac{\beta_{1,\log(O_d)}}{\beta_{1,up}}\right) up(d)+\left(\beta_{0,\log(O_d)}-\beta_{0,up}\frac{\beta_{1,\log(O_d)}}{\beta_{1,up}}+\max_{d\leq D}(\log(O_d)-\widehat{up}(d))\right),
\end{equation} in order to highlight the following relationship between $\widehat{\widehat{up}}(d)$ and $up(d)$
\begin{equation}\label{relUP}
 \widehat{\widehat{up}}(d)\equiv\widehat{\widehat{up}}(d,D):=F(D) \, up(d)+ G(D)   
\end{equation}
with $D=1100$ and 
$$F(D):=\frac{\beta_{1,\log(O_d)}}{\beta_{1,up}},\quad G(D):=\beta_{0,\log(O_d)}-\beta_{0,up}\frac{\beta_{1,\log(O_d)}}{\beta_{1,up}}+\max_{d\leq D}(\log(O_d)-\widehat{up}(d)),$$
where the dependence on $D$ is also in the calculation of the parameters $\beta_{0,up}$, $\beta_{1,up}$, $\beta_{0,\log(O_d)}$, $\beta_{1,\log(O_d)}.$
Then, using the numerical values in Tables \ref{tab:table1} and \ref{tab:table2}, we finally obtain the approximate expression
\begin{equation}\label{eq:calibration}
\widehat{\widehat{up}}(d)\approx 0.2256028352\ up(d)+4.574661581,
\end{equation}
which assumes values much closer to those of $\log(O_d)$ than $up(d)$ for $1\leq d\leq 1100$.
We stress that the identities in Table \ref{tab:table2} are exact for the function defined in \eqref{trasfbest}, from which 
$$
\widehat{\widehat{up}}(d)=(up(d)-\beta_{0,up})\frac{\beta_{1,\log(O_d)}}{\beta_{1,up}}+\beta_{0,\log(O_d)}+\max_{d\leq D}(\log(O_d)-\widehat{up}(d)),$$
with $\max_{d\leq D}(\log(O_d)-\widehat{up}(d))=-\min_{d\leq D}(\widehat{up}(d)-\log(O_d))=0.840217291747,$ while small discrepancies may appear if one recomputes it using the rounded coefficients in \eqref{eq:calibration}.

Formula \eqref{eq:calibration} naturally suggests that, on the finite range $1\leq d\leq 1100$, the effective coefficient in the upper bound from \eqref{eq:bound completo} may be much smaller than $\sqrt{2}$, provided that one inserts an additional additive term, namely $4.574661581$ as in \eqref{eq:calibration}.

In Figures \ref{step2} and \ref{step3bis} R-plots of the linear functions $L_{up(d)}$, $L_{\log(O_d)}$ and of the discrete functions $\log(O_d), up(d), \widehat{up}(d), \widehat{\widehat{up}}(d)$ are shown.

\begin{table}[htbp]
\centering
\caption{Coefficients of fitting lines of $\log(O_d)$ and $up(d)$ for the range $1\leq d\leq D=1100$ and slope of the empirical fit for $\log(O_d)$}
\label{tab:table1}
\begin{tabular}{l l}
\hline\\[-8pt]
$\beta_{0,\log(O_d)} = 19.12701407$ & $\beta_{1,\log(O_d)} = 0.07467593363$ \\[3pt]
\hline\\[-8pt]
$\beta_{0,up} = 68.22861859$ & $\beta_{1,up} = 0.3310061843$\\[3pt] \hline\\[-8pt]
$\beta_{1,\log(O_d)}/\beta_{1,up} = 0.2256028352$ & \\[3pt]
\hline
\end{tabular}
\end{table}

\begin{table}[htbp]
\centering
\caption{Discrepancies between $\log(O_d)$ and $up(d)$, $\widehat{up}(d)$, $\widehat{\widehat{up}}(d)$, for $1\leq d\leq D=1100$}
\label{tab:table2}
\begin{tabular}{l l l}
\hline\\[-8pt]
$\min_d(up(d) - \log(O_d))$ & $=$ &$ 0.9760727451$ \\[3pt]
\hline\\[-8pt]
$\min_d(\widehat{up}(d)-\log(O_d))$ &$=$ &$-0.840217291747$ \\[3pt]
\hline\\[-8pt]
$\min_d(\widehat{\widehat{up}}(d)-\log(O_d))$ & $=$  &$0$ \\[3pt]
\hline
\hline\\[-8pt]
$\max_d \vert up(d) - \log(O_d)\vert$      &$=$ & $313.681462325$ \\[3pt]
\hline\\[-8pt]
$\max_d \vert\widehat{up}(d)-\log(O_d)\vert$ &$=$ &$ 4.4289158466$ \\[3pt]
\hline\\[-8pt]
$\max_d\left(\widehat{\widehat{up}}(d)-\log(O_d)\right)$ &$=$ &$5.26913313834$ \\[3pt]
\hline
\end{tabular}
\end{table}

\begin{remark}
We also tested linear fits in the variable $x(d):={\sqrt{d}}\log(d)$, but this did not improve the quality of the approximation while making the graphical comparison less transparent, so we decided not to pursue this option. However, although a linear fit in $x(d)$ does not improve the approximation on our finite range, it yields compatible estimates for the effective constant in the $\sqrt{d}\log(d)$ scale.
\end{remark}


\begin{figure}[h]
	\centering
	\includegraphics[scale=0.49]{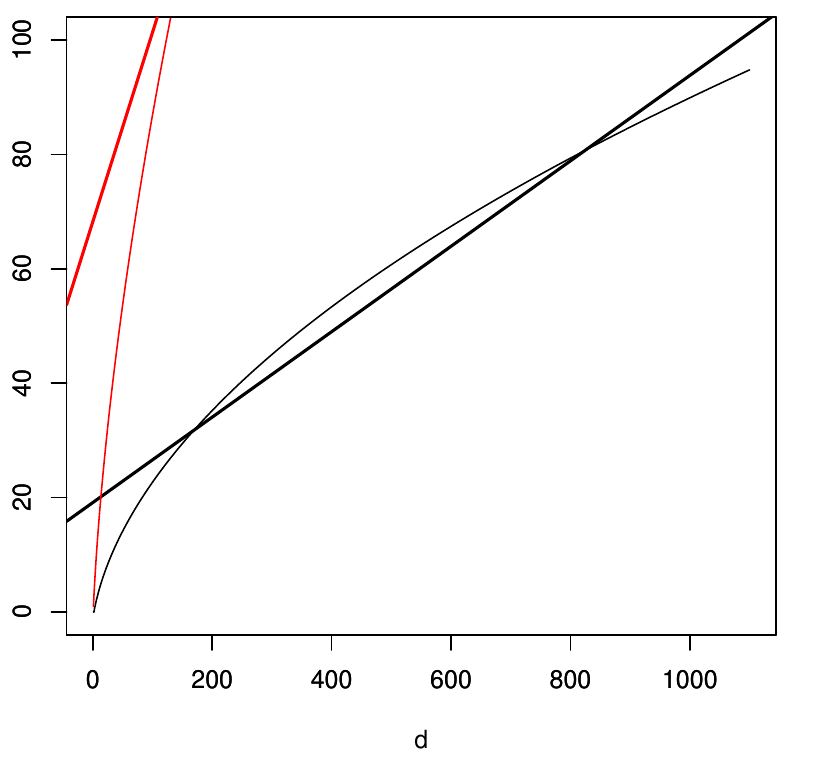}
	\caption{$1\leq d\leq D=1100$. In black, the plotted values $\log(O_d)$ by the curve and the plotted values of the function $L_{\log(O_d)}(d)$. In red, the values $up(d)$ and the corresponding values of the  function $L_{up}(d)$.}\label{step2}
\end{figure}
		

			\begin{figure}[h]
			\centering
			\includegraphics[scale=0.49]{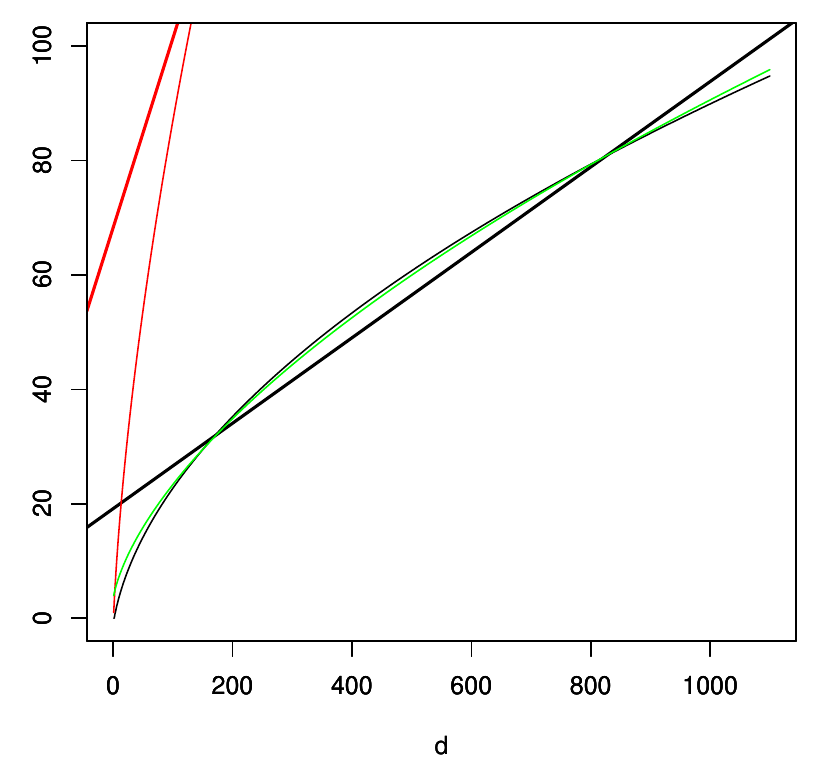}
			\includegraphics[scale=0.49]{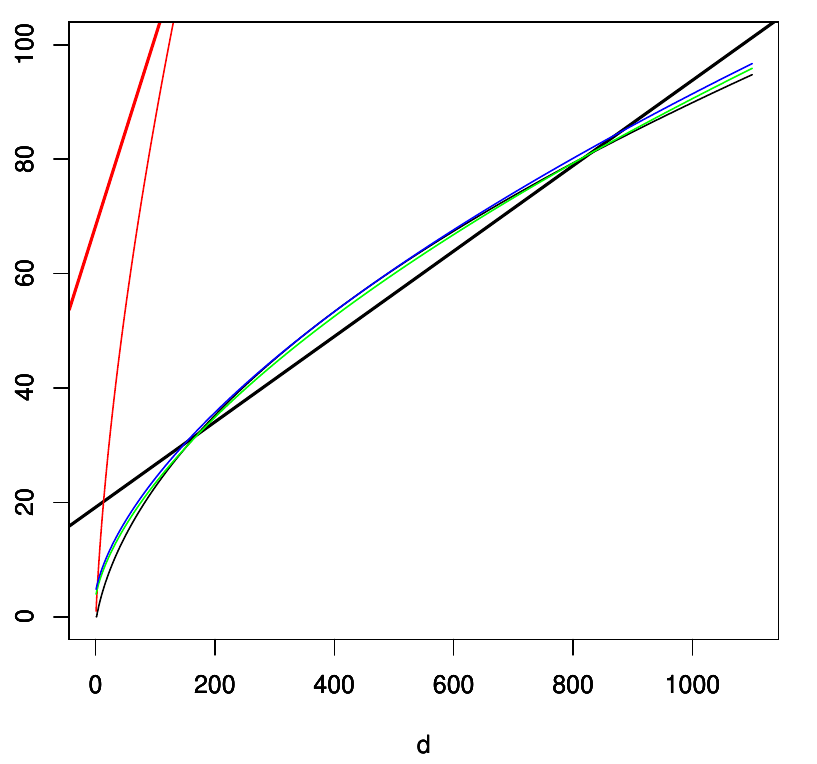}
			\caption{$1\leq d\leq D=1100$. LEFT: beyond the black and red lines of Fig.~\ref{step2}, the values $\widehat{up}(d)$, as given in~\eqref{trasf}, are plotted in green. We have  $\max_d\vert \widehat{up}(d)-\log(O_d)\vert=4.4289158466$, while $\max_d\vert{up(d)}-\log(O_d)\vert= 313.681462325$.
				RIGHT: beyond all previous plots, the values $\widehat{\widehat{up}}(d)$, as given in~\eqref{trasfbest}, are plotted in blue. Moreover, we obtain the value $\max_d(\log(O_d)-\widehat{up}(d))=0.840217291747$, whereas $\max_d\left(\log(O_d)-\widehat{\widehat{up}}(d)\right)=0$. Finally, $\max_d\left(\widehat{\widehat{up}}(d)-\log(O_d)\right)=5.26913313834$.}\label{step3bis}
		\end{figure}


\medskip
\noindent{\bf Case D=1100: lower bound.}
Similarly, we apply the strategy proposed above also to the explicit part of the asymptotic lower bound in \eqref{eq:bound completo}: 
	\begin{equation} \label{LOWB}
	low(d):=\pi\sqrt{\frac{2(d-1)}{3}}
-\log\!\bigl(4(d-1)\sqrt3\bigr).
	\end{equation} 
Here we consider the finite range $2\leq d\leq 1100$, because the function $low(d)$ is not defined on $d=1,$ and, for some specific comparisons, we will consider $7<d\leq 1100$ because in $1<d\leq 7$  some irrelevant exceptions occur. We consider the following discrete function
\begin{equation}\label{trasf_lo}
		\widehat{low}(d):=(low(d)-\beta_{0,low})\frac{\beta_{1,\log(O_d)}}{\beta_{1,low}}+\beta_{0,\log(O_d)} \quad \text{ (affine rescaling)},
	\end{equation}
where the symbols $\beta_{0,low}$, $\beta_{1,low}$, $\beta_{0,\log(O_d)}$, $\beta_{1,\log(O_d)}$ are $\beta_{0,f}$ and $\beta_{1,f}$ with, respectively, $f(d)=low(d)$ and $f(d)=\log(O_d)$, referring to the least-squares function $L_{f}(d) = \beta_{0,f} + \beta_{1,f} d$ that best fits the data points $(d,f(d))$ in the interval $2\leq d\leq D$. 

Even in this case, the construction \eqref{trasf_lo} of $\widehat{low}(d)$ does not exclude that there exists a positive integer~$d$ such that $\widehat{low}(d)-\log(O_d)>0$, namely $\widehat{low}(d)>\log(O_d)$ (see the second line of Table \ref{tab:table2_lo}). Hence, we also consider the function
		\begin{equation}\label{trasfbest_lo}
		\widehat{\widehat{low}}(d):= \widehat{low}(d) - \max_{d\leq D}(\widehat{low}(d)-\log(O_d)),
	\end{equation}
because by definition it satisfies $\log(O_d)\geq \widehat{\widehat{low}}(d)$, for every integer $2\leq d\leq D$, with  $-\max_{d\leq D}(\widehat{low}(d)-\log(O_d))=\min_{d\leq D}(\log(O_d)-\widehat{low}(d))$.

{\em We highlight that both the functions $\widehat{low}(d)$ and $\widehat{\widehat{low}}(d)$ depend on the value of $D$, although we are choosing a notation that does not make this dependence explicit. When necessary, we will make this dependence explicit even in the notation.}

The function~\eqref{trasfbest_lo} turns out to be a better empirical lower bound for $\log(O_d)$ than the function $low(d)$ in the interval $7< d\leq D=1100$, analogously to the above upper bound case.

Using \eqref{trasf_lo} and \eqref{trasfbest_lo}, we can write 
\begin{equation}\label{linear_transLOWER}
\widehat{\widehat{low}}(d)= \left(\frac{\beta_{1,\log(O_d)}}{\beta_{1,low}}\right) low(d)+\left(\beta_{0,\log(O_d)}-\beta_{0,low}\frac{\beta_{1,\log(O_d)}}{\beta_{1,low}}- \max_{d\leq D}(\widehat{low}(d)-\log(O_d))\right),
\end{equation}
and using the numerical values in Tables \ref{tab:table1_lo} and \ref{tab:table2_lo}, we obtain the approximate expression
\begin{equation}\label{eq:calibration_lo}
\widehat{\widehat{low}}(d)\approx 1.26150754299\ low(d)-1.30986731154,
\end{equation}
which takes values closer to those of $\log(O_d)$ than $low(d)$ for $7 < d\leq 1100$.
We stress that the identities in Table \ref{tab:table2_lo} are exact for the function defined in \eqref{trasfbest_lo}, from which 
$$
\widehat{\widehat{low}}(d)=(low(d)-\beta_{0,low})\frac{\beta_{1,\log(O_d)}}{\beta_{1,low}}+\beta_{0,\log(O_d)}- \max_{d\leq D}(\widehat{low}(d)-\log(O_d)),$$
with $\max_{d\leq D}(\widehat{low}(d)-\log(O_d))=0.0554619907174,$
while small discrepancies may appear if one recomputes it using the rounded coefficients in \eqref{eq:calibration_lo}.

\begin{table}[htbp]
\centering
\caption{Coefficients of fitting lines of $\log(O_d)$ and $low(d)$ for the range $2 \leq  d\leq D=1100$ and slope of the empirical fit for $\log(O_d)$}
\label{tab:table1_lo}
\begin{tabular}{l l}
\hline \\[-8pt]
$\beta_{0,\log(O_d)} = 19.1970929$ & $\displaystyle{\beta_{1,\log(O_d)} = 
0.0745804584}$ \\[3pt]
\hline \\[-8pt]
$\beta_{0,low} = 
16.2119508
$ & $\beta_{1,low} = 
0.0591201050$ \\[3pt] \hline \\[-8pt]
$\beta_{1,\log(O_d)}/\beta_{1,low} = 
1.26150754299$ & \\[3pt]
\hline
\end{tabular}
\end{table}

\begin{table}[htbp]
\centering
\caption{Discrepancies between $\log(O_d)$ and $low(d)$, $\widehat{low}(d)$, $\widehat{\widehat{low}}(d)$, for $7< d\leq D=1100$}
\label{tab:table2_lo}
\begin{tabular}{l l l}
\hline \\[-8pt]
$\min_d (\log(O_d)-low(d)  )$ & $=$ &$ 0.0556814562012$ \\[3pt]
\hline \\[-8pt]
$\min_d ( \log(O_d) -\widehat{low}(d))$ &$=$ &$-0.0554619907174$ \\[3pt]
\hline \\[-8pt]
$\min_d \left(\log(O_d)-\widehat{\widehat{low}}(d)\right)$ & $=$  &$0
$ \\[3pt]
\hline
\hline \\[-8pt]
$\max_d \vert \log(O_d)-low(d) \vert$      &$=$ & $18.677299565

$ \\[3pt]
\hline \\[-8pt]
$\max_d \vert \log(O_d)-\widehat{low}(d) \vert$ &$=$ &$   0.462761696891$ \\[3pt]
\hline \\[-8pt]
$\max_d\left(\log(O_d)- \widehat{\widehat{low}}(d)\right)$ &$=$ &$ 0.518223687609
$ \\[3pt]
\hline
\end{tabular}
\end{table}

In Figure \ref{figlower} plots in R of the discrete functions $\log(O_d), low(d), \widehat{low}(d), \widehat{\widehat{low}}(d)$ are shown. We remark that even if the better approximation provided by $\widehat{\widehat{low}}(d)$ than that of $\widehat{low}(d)$ is not clearly visible in the figures, the improvement is evident from comparing the second and the third line of Table~\ref{tab:table2_lo}.


		\begin{figure}[h]
			\centering
			\includegraphics[scale=0.32]{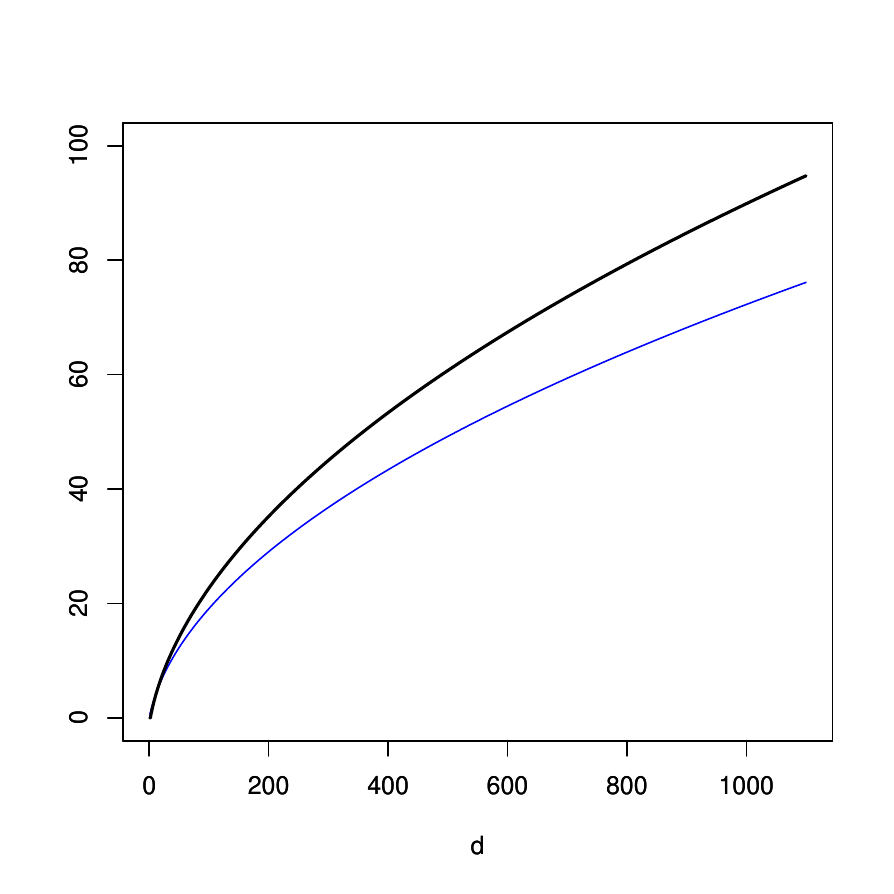}
			\includegraphics[scale=0.32]{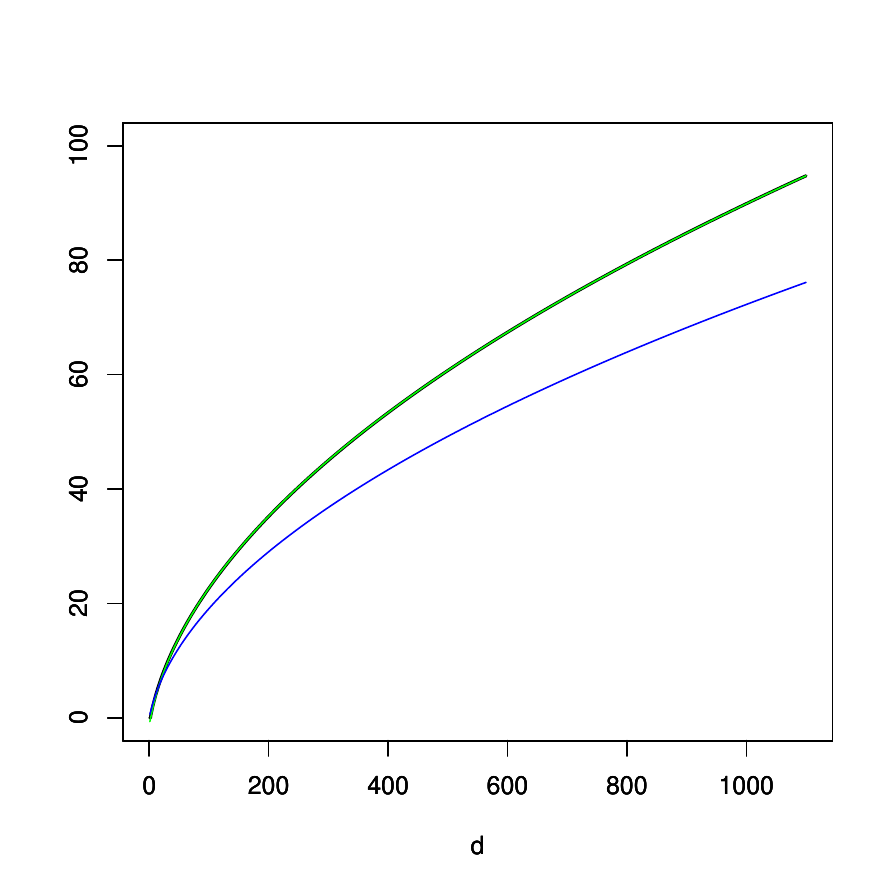}
\includegraphics[scale=0.32]{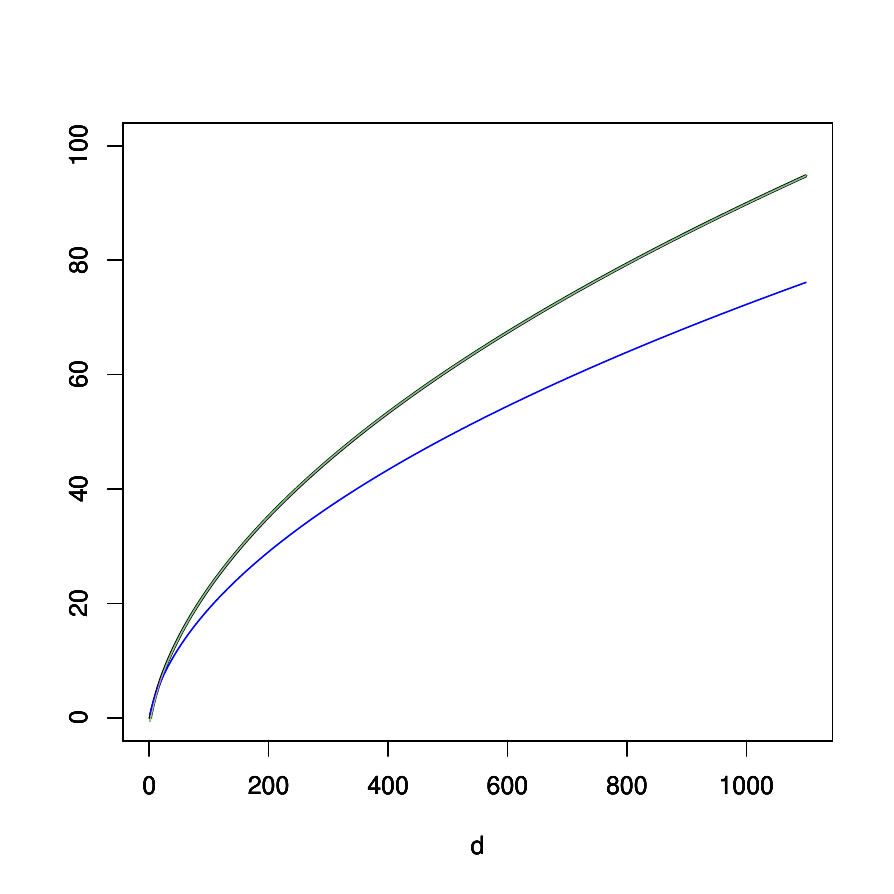}
            \caption{In blue the plot of the function $low(d)$, for $2\leq d\leq D=1100$. In the same finite range, the plot of the values of $\log(O_d)$ (black) on the left, the plot of the values of $\widehat{low}(d)$ (green) in the middle, and the plot of the values of $\widehat{\widehat{low}}(d)$ (magenta) on the right. Note that in all sub-figures we plotted $\log(O_d)$ (black), even if in the middle and in the right it is overwritten by $\widehat{low}(d)$ and $\widehat{\widehat{low}}(d)$, respectively.}
            \label{figlower}
		\end{figure}

\subsection{Empirical prediction estimates for upper and lower bounds}

In order to construct a possible prediction for the upper bound $\widehat{\widehat{up}}(d)$ for $d>1100,$ we follow the following strategy.  From \eqref{linear_transUP}, considering only the first 250 values of $d,$ we construct:
\begin{equation}\label{linear_transUP250}
\widehat{\widehat{up}}_{250}(d):= \left(\frac{\beta_{1,\log(O_d)}}{\beta_{1,up}}\right) up(d)+\left(\beta_{0,\log(O_d)}-\beta_{0,up}\frac{\beta_{1,\log(O_d)}}{\beta_{1,up}}+\max_{d\leq D}(\log(O_d)-\widehat{up}(d))\right),
\end{equation}
where $D=250$ and the calculation of the parameters $\beta_{0,up}$, $\beta_{1,up}$, $\beta_{0,\log(O_d)}$, $\beta_{1,\log(O_d)}$ refers to the interval $1\leq d \leq 250.$
Hence, for $1\leq d \leq D=250$,  we obtain
\begin{equation}\label{linear_transUP250values}
\widehat{\widehat{up}}_{250}(d)=  0.252677948393 \, 
 up(d) +0.733721091341.
\end{equation}
Then, we provide a prediction estimate of the upper bound by evaluating  $\widehat{\widehat{up}}_{250}(d)$ for successive values of $d,$ i.e. $251\leq d \leq 1100$. In order to verify graphically the proximity of this predicted estimated upper bound $\widehat{\widehat{up}}_{250}(d)$ to the values of $\log(O_d),$ we plot it (in orange) on left of Figure \ref{figpredictionUP}.

\begin{figure}[h]
			\centering
			\includegraphics[scale=0.49]{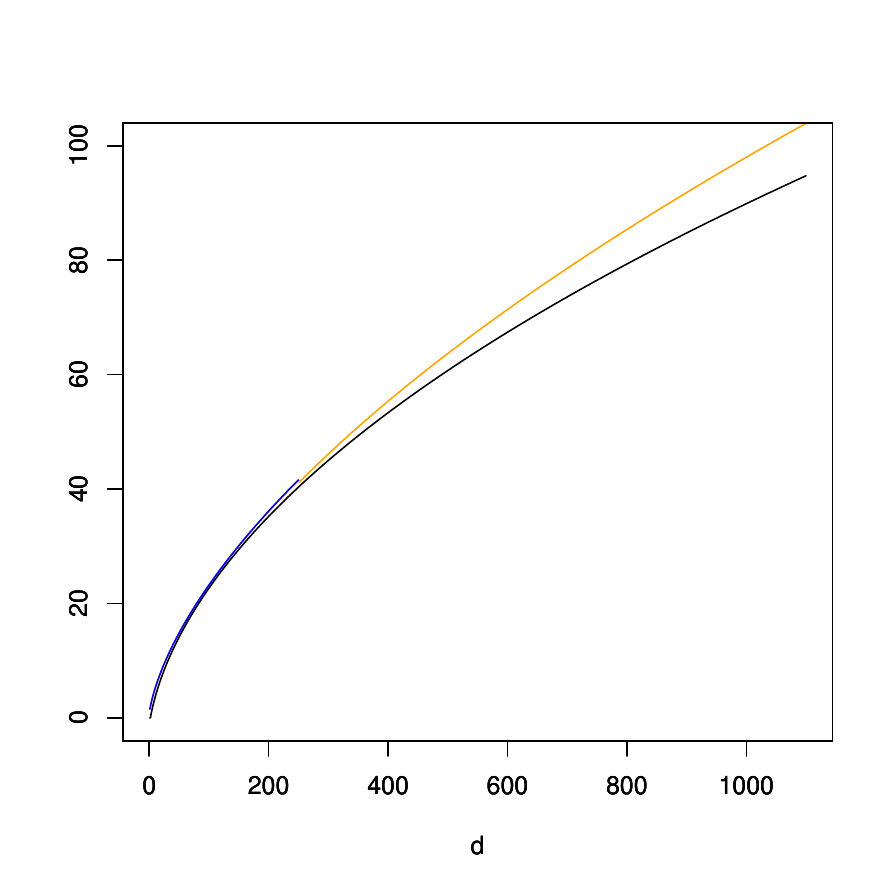}
    \includegraphics[scale=0.49]{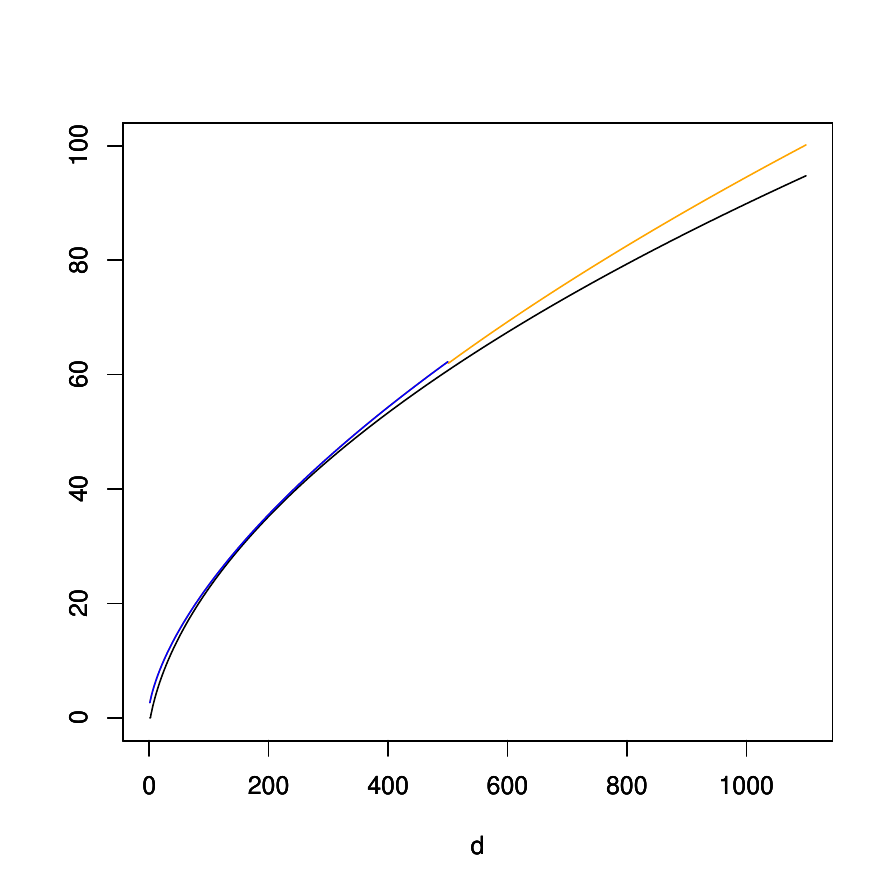}
            \caption{Prediction estimates of upper bound. LEFT: $\widehat{\widehat{up}}_{250}(d)$ as in \eqref{linear_transUP250values} (in blue) for $1\leq d\leq 250$, and its prediction estimate (in orange) for $251\leq d\leq 1100$; RIGHT: $\widehat{\widehat{up}}_{500}(d)$ as in \eqref{linear_transUP500values} (in blue) for $1\leq d\leq 500$, and its prediction estimate (in orange) for $501\leq d\leq 1100$. Note that $\max_{1\leq d\leq 1100}\left(\widehat{\widehat{up}}_{250}(d)-\log(O_d)\right)=9.16621718467$, whereas $\max_{1\leq d\leq 1100}\left(\widehat{\widehat{up}}_{500}(d)-\log(O_d)\right)=5.39351115845.$}
            \label{figpredictionUP}
		\end{figure}
        
\begin{figure}[h]
			\centering
			\includegraphics[scale=0.49]{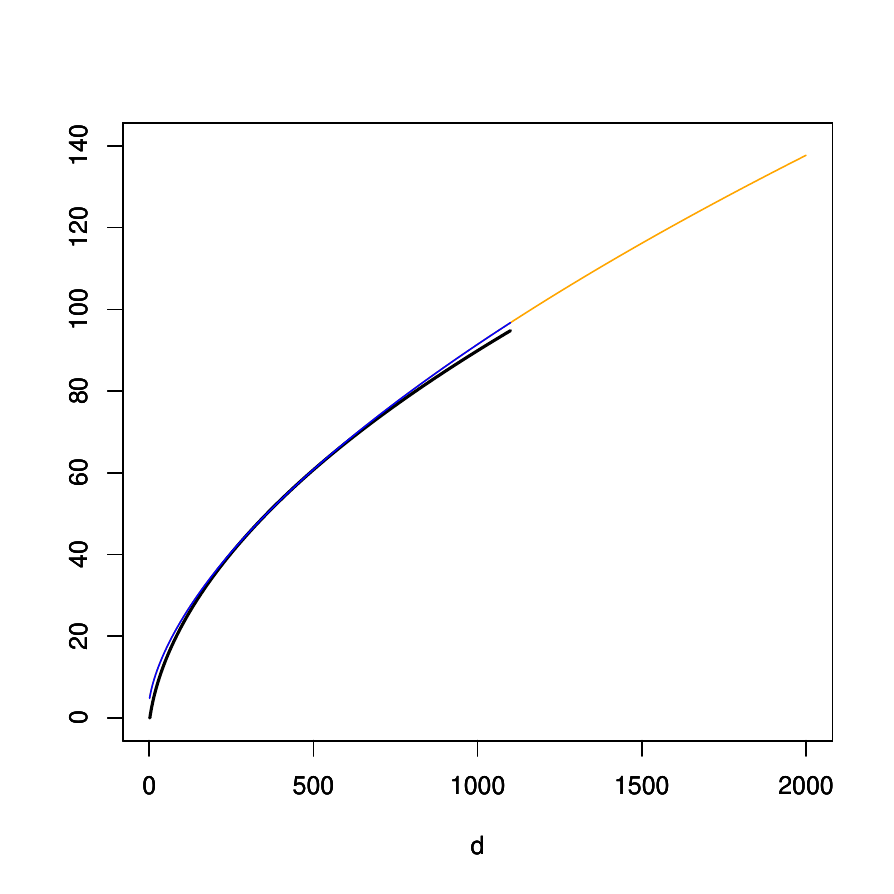}
    \includegraphics[scale=0.49]{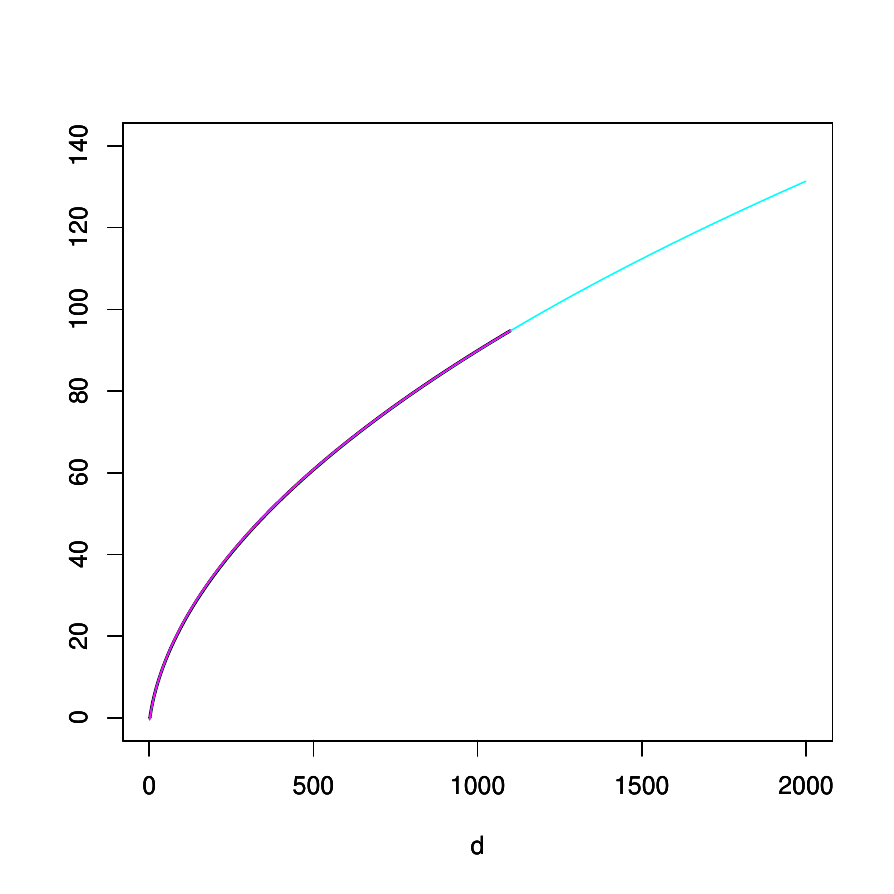}
            \caption{Prediction estimates until $d=2000.$ LEFT: $\widehat{\widehat{up}}(d)$ (in blue) as in  \eqref{eq:calibration} evaluated for $1\leq d\leq 1100$, and its prediction estimate (in orange) as in  \eqref{eq:calibration} evaluated for $1101\leq d\leq 2000$. RIGHT: $\widehat{\widehat{low}}(d)$ (in magenta) as in  \eqref{eq:calibration_lo} evaluated for $2\leq d\leq 1100$, and its prediction estimate  (in cyan) as in  \eqref{eq:calibration_lo} evaluated for $1101\leq d\leq 2000$. In both sides the plot of $\log(O_d)$ (in black)  for $1\leq d\leq 1100$ is also provided.}
            \label{figpredictionUPandLOWER}
		\end{figure}
        
\begin{figure}[h]
			\centering
    \includegraphics[height=7.6cm]{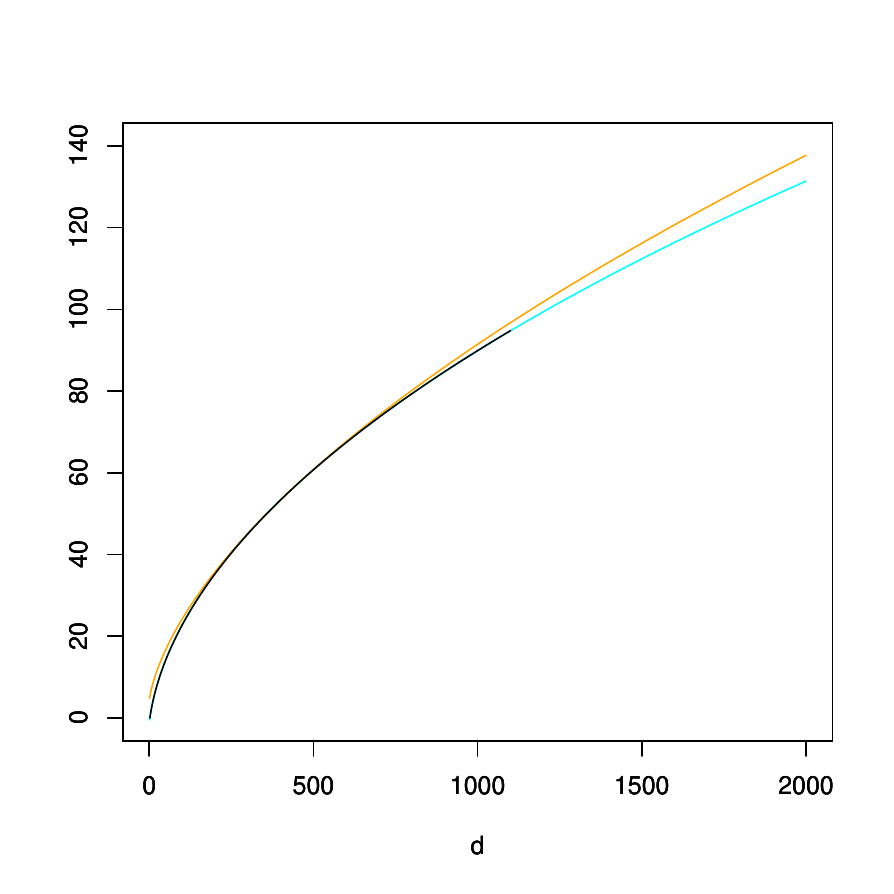}
\includegraphics[height=6.65cm,width=0.475\textwidth]{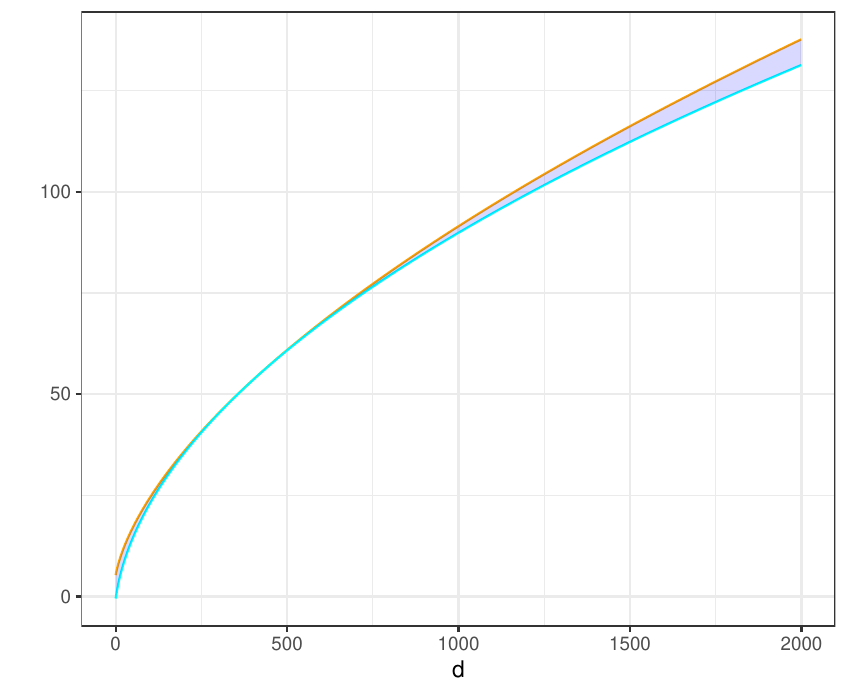}
\caption{Prediction zone until $d=2000.$ LEFT: first, for $1\leq d \leq D=1100,$  $\widehat{\widehat{up}}(d)$ (in orange) as in  \eqref{eq:calibration} and $\widehat{\widehat{low}}(d)$ (in cyan) as in  \eqref{eq:calibration_lo}. Then, the prediction zone delimited by the prediction estimate of the empirical calibration  $\widehat{\widehat{up}}(d)$ (in orange) as in  \eqref{eq:calibration} and by the prediction estimate of the empirical calibration $\widehat{\widehat{low}}(d)$ (in cyan) as in  \eqref{eq:calibration_lo}, both evaluated for $1101\leq d \leq 2000.$ The plot of $\log(O_d)$ (in black)  for $1\leq d\leq 1100$ is also provided. The maximum width of the prediction zone is given by means of  $\max_{2\leq d \leq 2000}(\widehat{\widehat{up}}(d)-\widehat{\widehat{low}}(d))=6.31289617142$. RIGHT: The same, but without black plot of $\log(O_d)$ and with shaded predicted zone.}
        \label{figpredictionStrip}
		\end{figure}
In order to show how the prediction estimate of the upper bound improves with increasing the size of the data sample, we apply the same procedure again for values $1\leq d \leq 500.$ Specifically, again from \eqref{linear_transUP}, considering the first 500 values of $d,$  we construct for $1\leq d \leq D=500:$
\begin{equation}\label{linear_transUP500values}
\widehat{\widehat{up}}_{500}(d) :=   0.240139019901 \, 
 up(d) +2.08262890146.
\end{equation}
We provide a second prediction estimate of the upper bound by evaluating  $\widehat{\widehat{up}}_{500}(d)$ for subsequent values of $d,$ i.e. $501\leq d \leq 1100$. It is possible to compare graphically the prediction estimates for the  upper bound $\widehat{\widehat{up}}_{500}(d)$ plotted (in orange) on right of Figure \ref{figpredictionUP} and $\widehat{\widehat{up}}_{250}(d)$ plotted (in orange) on left of Figure \ref{figpredictionUP}. 

Finally, by considering that
$$\max_{1\leq d\leq 1100}\left(\widehat{\widehat{up}}_{250}(d)-\log(O_d)\right)=9.16621718467> \max_{1\leq d\leq 1100}\left(\widehat{\widehat{up}}_{500}(d)-\log(O_d)\right)=5.39351115845$$
we see that the prediction estimate $\widehat{\widehat{up}}_{500}(d)$ is better than $\widehat{\widehat{up}}_{250}(d).$
From the last line of Table~\ref{tab:table2} we also highlight that $\max_{1\leq d\leq 1100}\left(\widehat{\widehat{up}}_{1100}(d)-\log(O_d)\right)=5.26913313834$, so that this difference appears to decrease, but more slowly, suggesting a stabilization.

Consequently, our preferred empirical prediction estimate is obtained by evaluating $\widehat{\widehat{up}}(d)$ as in~\eqref{eq:calibration} for successive values of $d$, for instance for $1101\leq d \leq 2000.$ 

We apply an analogous strategy also to the lower bound $\widehat{\widehat{low}}(d)$ as in \eqref{eq:calibration_lo} by evaluating it on $1101\leq d \leq 2000.$ The plots of these prediction estimates are in Figure \ref{figpredictionUPandLOWER}.

Finally, in Figure \ref{figpredictionStrip}, we are able to provide the plot of the zone confined by the prediction estimate of the upper bound $\widehat{\widehat{up}}(d)$ (in orange) as in  \eqref{eq:calibration} and by the prediction estimate of the lower bound $\widehat{\widehat{low}}(d)$ (in cyan) as in  \eqref{eq:calibration_lo}, both evaluated for $1101\leq d \leq 2000.$ The quantitative information  $$\max_{2\leq d \leq 2000}(\widehat{\widehat{up}}(d)-\widehat{\widehat{low}}(d))=6.31289617142$$
can be used as a measure of the width of the prediction zone, but it is also the maximum width of the prediction zone of the provided estimates.

Since the values of $O_d$ are not known in the range $1101 \leq d \leq 2000$, the curves of Figures~\ref{figpredictionUPandLOWER} and~\ref{figpredictionStrip} should be interpreted as empirical prediction estimates rather than as rigorously proved upper and lower bounds. Nevertheless, the validation tests on the known range $1\le d\le 1100$, together with the stability observed when increasing the calibration range from $250$ to $500$ and then to $1100$, provide evidence that these estimates should remain close to the actual values of $\log(O_d)$.

\section{Conclusions}

In this paper, starting from an iterative formula for the computation of $O_d$, we established that the sequence $(A_{d+2})_{d\geq 1}$ is sub-Fibonacci. This result gave an enhancement of the sub-Fibonacci behavior of $(O_d)_{d\geq 1}$. We then translated the iterative formula into an effective and memory-efficient algorithm, which allowed the computation of $O_d$ for all $1\leq d\leq 1100$. The resulting data set made it possible to compare the observed values of $\log(O_d)$ with the asymptotic upper and lower bounds of Stanley and Zanello, and to propose empirical finite-range  calibrations, in the interval $1\leq d\leq 1100$, that better reflect the behavior of the data while preserving the expected growth scale. These results provide new quantitative information on finite O-sequences in a finite range. The method can also be  used in all intervals where the values of $O_d$ are known. Moreover, as a consequence of the strong Cesàro convergence of the rate sequence $(O_d/O_{d-1})_{d\geq 2}$ to $1$, we observe that the Stanley-Zanello upper bound implies a negative answer to a question formulated by Roberts in \cite{RR}, under the condition that $(O_d/O_{d-1})_{d\geq 2}$ converges in the usual sense.


\section*{Acknowledgements}
The first author is a member of GNSAGA (INdAM, Italy). The third author is a member of GNCS (INdAM, Italy).



\providecommand{\bysame}{\leavevmode\hbox to3em{\hrulefill}\thinspace}
\providecommand{\MR}{\relax\ifhmode\unskip\space\fi MR}
\providecommand{\MRhref}[2]{%
  \href{http://www.ams.org/mathscinet-getitem?mr=#1}{#2}
}
\providecommand{\MRhref}[2]{#2}

\end{document}